\documentclass[10pt,psamsfonts]{amsart}
\usepackage{amsmath}
\usepackage{amsthm}
\usepackage{amssymb}
\usepackage{amscd}
\usepackage{amsfonts}
\usepackage{amsbsy}
\usepackage{graphicx}
\usepackage[dvips]{psfrag}

\newcommand{\R}{\ensuremath{\mathbb{R}}}

\newcommand{\CC}{\mathcal{C}}

\newcommand{\CG}{\ensuremath{\mathcal{G}}}

\newcommand{\T}{\theta}

\newcommand{\al}{\alpha}
\newcommand{\be}{\beta}
\newcommand{\cl}{\mbox{\normalfont{Cl}}}

\newcommand{\A}{\ensuremath{\mathcal{A}}}
\newcommand{\B}{\ensuremath{\mathcal{B}}}
\newcommand{\C}{\ensuremath{\mathcal{C}}}
\newcommand{\D}{\ensuremath{\mathcal{D}}}

\def\p{\partial}
\def\e{\varepsilon}

\newtheorem {theorem} {Theorem} %[section]

\newtheorem {corollary} [theorem] {Corollary}
\newtheorem {lemma} [theorem] {Lemma}

\begin{document}

\title[On the periodic solutions of a perturbed double pendulum]
{On the periodic solutions\\ of a perturbed double pendulum}

\author[J. Llibre, D.D. Novaes and M.A. Teixeira]
{Jaume Llibre$^1$, Douglas D. Novaes$^2$  and Marco Antonio
Teixeira$^2$}

\address{$^1$ Departament de Matematiques,
Universitat Aut\`{o}noma de Barcelona, 08193 Bellaterra, Barcelona,
Catalonia, Spain} \email{jllibre@mat.uab.cat}

\address{$^2$ Departamento de Matematica, Universidade
Estadual de Campinas, Caixa Postal 6065, 13083--859, Campinas, SP,
Brazil} \email{ddnovaes@gmail.com, teixeira@ime.unicamp.br}


\date{}
\dedicatory{To Waldyr Oliva in his 80$^{th}$ Birthday}

\maketitle

\begin{abstract}
We provide sufficient conditions for the existence of periodic
solutions of the planar perturbed double pendulum with small
oscillations having equations of motion
\[
\begin{array}{l}
\ddot{\T}_{1}=-2a\T_{1}+a\T_{2}+\e F_1(t,\T_1,\dot \T_1,\T_2, \dot \T_2),\\
\ddot{\T}_{2}=2a\T_{1}-2a\T_{2}+\e F_2(t,\T_1,\dot \T_1,\T_2, \dot
\T_2),
\end{array}
\]
where $a$ and $\e$ are real parameters. The two masses of the
unperturbed double pendulum are equal, and its two stems have the
same length $l$. In fact $a=g/l$ where $g$ is the acceleration of
the gravity. Here the parameter $\e$ is small and the smooth
functions $F_1$ and $F_2$ define the perturbation which are periodic
functions in $t$ and in resonance $p$:$q$ with some of the periodic
solutions of the unperturbed double pendulum, being $p$ and $q$
positive integers relatively prime.
\end{abstract}

\section{Introduction and statement of the main results}\label{s1}

We consider a system of two point masses $m_1$ and $m_2$ moving in a
fixed plane, in which the distance between a point (called pivot)
and $m_1$ and the distance between $m_1$ and $m_2$ are fixed, and
equal to $l_1$ and $l_2$ respectively. We assume the masses do not
interact. We allow gravity to act on the masses $m_1$ and $m_2$.
This system is called the planar {\it double pendulum}.

\smallskip

The position of the double pendulum is determined by the two angles
$\T_1$ and $\T_2$ shown in Figure 1. We consider only the motion in
the vicinity of the equilibrium $\T_1= \T_2= 0$, i.e. we are only
interested in small oscillations around this equilibrium. Expanding
the Lagrangian of this system to second order in $\T_1$ and $\T_2$
and their time derivatives, the corresponding Lagrange equations of
motion are
\begin{equation}\label{pendulum}
\begin{array}{r}
(m_{1}+m_{2})l_{1}\ddot{\theta}_{1}+m_{2}l_{2}\ddot{\theta}_{2}+
(m_{1}+m_{2})g\theta_{1}=0, \vspace{0.2cm}\\
m_{2}l_{2}\ddot{\theta}_{2}+m_{2}l_{1}\ddot{\theta}_{1}+m_{2}g\theta_{2}=0.
\end{array}
\end{equation}
For more details on these equations of motion see \cite{Ir}. Here
the dot denotes derivative with respect to the time $t$.

\begin{figure}
\psfrag{A}{$l_{1}$} \psfrag{B}{$\theta_{1}$}
\psfrag{C}{$\theta_{2}$} \psfrag{D}{$l_{2}$} \psfrag{E}{$m_{1}$}
\psfrag{F}{$m_{2}$} \psfrag{G}{$g$}
\includegraphics[width=3cm]{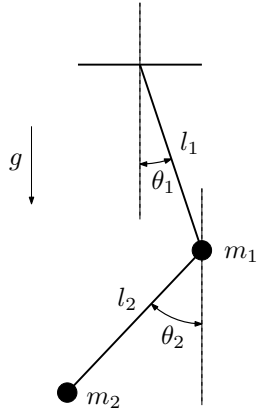}
\vskip 0cm \centerline{} \caption{\small \label{fig} The planar
double pendulum.}
\end{figure}

\smallskip

If we take $m_{1}=m_{2}=m>0$, $l_{1}=l_{2}=l>0$ and divide the above
equations by $ml$, system (\ref{pendulum}) becomes
\begin{equation}\label{pendulum2}
\begin{array}{r}
2\, \ddot{\theta}_{1}+\ddot{\theta}_{2}+2\dfrac{g}{l}\theta_{1}=0,
\vspace{0.2cm}\\
\ddot{\theta}_{1}+\ddot{\theta}_{2}+\dfrac{g}{l}\theta_{2}=0.
\end{array}
\end{equation}

Isolating $\ddot{\theta}_{1}$ and $\ddot{\theta}_{2}$ of equations
(\ref{pendulum2}) and denoting $a=g/l$, we obtain that the equations
of motion of the double pendulum with small oscillations, two equal
masses $m$ and having the two stems the same length $l$ are
\begin{equation}\label{e1}
\begin{array}{l}
\ddot{\T}_{1}=-2a\T_{1}+a\T_{2}, \vspace{0.2cm}\\
\ddot{\T}_{2}=2a\T_{1}-2a\T_{2},
\end{array}
\end{equation}
where $g$ is the acceleration of the gravity.

\smallskip

Taking a new time $\tau$ given by the rescaling $\tau= \sqrt{a}\, t$
the equations of motion \eqref{e1} become
\begin{equation}\label{e1a}
\begin{array}{l}
\T''_{1}=-2\T_{1}+\T_{2}, \vspace{0.2cm}\\
\T''_{2}=2\T_{1}-2\T_{2},
\end{array}
\end{equation}
where now the prime denotes derivative with respect to the new time
$\tau$.

\smallskip

The objective of this paper is to provide a system of nonlinear
equations whose simple zeros provide periodic solutions of the
perturbed planar double pendulum with equations of motion
\begin{equation}\label{e2}
\begin{array}{l}
\T^{''}_{1}=-2\T_{1}+\T_{2}+\e F_1(\tau,\T_1,\T'_1,\T_2,\T'_2), \vspace{0.2cm}\\
\T^{''}_{2}=2\T_{1}-2\T_{2}+\e F_2(\tau,\T_1,\T'_1,\T_2,\T'_2),
\end{array}
\end{equation}
where $\e$ is a small parameter. Here the smooth functions $F_1$ and
$F_2$ define the perturbation. These functions are periodic in
$\tau$ and in resonance $p$:$q$ with some of the periodic solutions
of the unperturbed double pendulum, being $p$ and $q$ positive
integers relatively primes. In order to present our results we need
some preliminary definitions and notations.

\smallskip

The unperturbed system \eqref{e1a} has a unique singular point, the
origin with eigenvalues
\[
\pm \sqrt{2-\sqrt{2}} \, i, \quad \pm \sqrt{2+\sqrt{2}}\, i.
\]
Consequently this system in the phase space $(\T_1,\T'_1,
\T_2,\T'_2)$ has two planes filled of periodic solutions with the
exception of the origin. These periodic solutions have periods
\[
T_1= \dfrac{2\pi}{\sqrt{2-\sqrt{2}}} \quad \mbox{or} \quad T_2=
\dfrac{2\pi}{\sqrt{2+\sqrt{2}}},
\]
according they belong to the plane associated to the eigenvectors
with eigenvalues $\pm \sqrt{2-\sqrt{2}} \, i$ or $\pm
\sqrt{2+\sqrt{2}} \, i$, respectively. We shall study which of these
periodic solutions persist for the perturbed system \eqref{e2} when
the parameter $\e$ is sufficiently small and the perturbed functions
$F_i$ for $i=1,2$ have period either $p T_1/q$, or $p T_2/q$, where
$p$ and $q$ are positive integers relatively prime.

\smallskip

We define the functions:
\begin{equation}\label{e3}
\begin{array}{l}
\CG_1(X_0,Y_0)= \displaystyle \int_0^{pT_1}
\sin\left(\sqrt{2-\sqrt{2}} \, \tau\right) \left(\sqrt{2}\, \bar
F_1+ \bar F_2\right) d\tau, \vspace{0.2cm}\\
\CG_2(X_0,Y_0)= \displaystyle  \displaystyle \int_0^{pT_1} \cos
\left(\sqrt{2-\sqrt{2}} \, \tau\right) \left(\sqrt{2}\, \bar F_1+
\bar F_2\right) d\tau,
\end{array}
\end{equation}
where
\begin{equation}\label{jj}
\bar F_k= F_k(\tau,A_1,B_1,C_1,D_1)
\end{equation}
for $k=1,2$ with
\[
\begin{array}{l}
A_1= \dfrac{1}{\sqrt{4-2 \sqrt{2}}} \left(X_0 \cos
\left(\sqrt{2-\sqrt{2}}\, \tau\right)+Y_0 \sin
\left(\sqrt{2-\sqrt{2}}\, \tau\right)\right), \vspace{0.2cm}\\
B_1= \dfrac{1}{\sqrt{2}} \left(Y_0 \cos \left(\sqrt{2-\sqrt{2}}\,
\tau\right)-X_0\sin \left(\sqrt{2-\sqrt{2}}\, \tau\right) \right), \vspace{0.2cm}\\
C_1= \dfrac{1}{\sqrt{2-\sqrt{2}}} \left(X_0 \cos
\left(\sqrt{2-\sqrt{2}}\, \tau\right)+Y_0
\sin\left(\sqrt{2-\sqrt{2}}\, \tau\right) \right), \vspace{0.2cm}\\
D_1= Y_0 \cos \left(\sqrt{2-\sqrt{2}}\, \tau\right)-X_0 \sin
   \left(\sqrt{2-\sqrt{2}}\, \tau\right).
\end{array}
\]

A zero $(X_0^*,Y_0^*)$ of the nonlinear system
\begin{equation}\label{e4}
\CG_1(X_0,Y_0)=0,\quad \CG_2(X_0,Y_0)=0,
\end{equation}
such that
\[
\det \left(\left.
\dfrac{\p(\CG_1,\CG_2)}{\p(X_0,Y_0)}\right|_{(X_0,Y_0)=
(X_0^*,Y_0^*)}\right) \neq 0,
\]
is called a {\it simple zero} of system \eqref{e4}.

\smallskip

Our main result on the periodic solutions of the perturbed double
pendulum \eqref{e2} which bifurcate from the periodic solutions of
the unperturbed double pendulum \eqref{e1a} with period $T_1$
traveled $p$ times is the following.

\begin{theorem}\label{t1}
Assume that the functions $F_k$ of the perturbed double pendulum
with equations of motion \eqref{e2} are periodic in $\tau$ of period
$pT_1/q$ with $p$ and $q$ positive integers relatively prime. Then
for $\e\neq 0$ sufficiently small and for every simple zero
$(X_0^*,Y_0^*)\neq (0,0)$ of the nonlinear system \eqref{e4}, the
perturbed double pendulum \eqref{e2} has a periodic solution
$(\T_1(\tau,\e),\T_2(\tau,\e))$ tending when $\e\to 0$ to the
periodic solution $(\T_1(\tau),\T_2(\tau))$ run $p$ times of the
unperturbed double pendulum \eqref{e1a} given by
\begin{equation}\label{e5}
\begin{array}{l}
\left( \dfrac{1}{\sqrt{4-2 \sqrt{2}}} \left(X_0^* \cos
\left(\sqrt{2-\sqrt{2}}\, \tau\right)+Y_0^* \sin
\left(\sqrt{2-\sqrt{2}}\, \tau\right)\right), \right. \vspace{0.2cm}\\
\left. \dfrac{1}{\sqrt{2-\sqrt{2}}} \left(X_0^* \cos
\left(\sqrt{2-\sqrt{2}}\, \tau\right)+Y_0^*
\sin\left(\sqrt{2-\sqrt{2}}\, \tau\right) \right) \right).
\end{array}
\end{equation}
\end{theorem}

Theorem \ref{t1} is proved in section \ref{s3}. Its proof is based
in the averaging theory for computing periodic solutions, see the
appendix.

\smallskip

We provide an application of Theorem \ref{t1} in the following
corollary, which will be proved in section \ref{s4}.

\begin{corollary}\label{c1}
If $F_2(\T_1,\T'_1,\T_2,\T'_2)= \left(1-\T_1^2\right) \sin
\left(\sqrt{2-\sqrt{2}}\, t\right)$ and $F_1(\T_1,\T'_1,\T_2,\T'_2)=
0$, then the differential equation \eqref{e2} for $\e\neq 0$
sufficiently small has two periodic solutions
$(\T_1(\tau,\e),\T_2(\tau,\e))$ tending when $\e\to 0$ to the two
periodic solutions $(\T_1(\tau),\T_2(\tau))$ of the unperturbed
double pendulum \eqref{e1a} given by \eqref{e5} with
\[
(X_0^*,Y_0^*)= \left(2 \sqrt{2 \left(2-\sqrt{2}\right)},0   \right)
\,\,\, \mbox{and} \,\,\, (X_0^*,Y_0^*)= \left(0, 2 \sqrt{\frac{2}{3}
\left(2-\sqrt{2}\right)} \right),
\]
respectively.
\end{corollary}

Now we define the functions:
\begin{equation}\label{e3a}
\begin{array}{l}
\CG^1(Z_0,W_0)= \displaystyle \int_0^{pT_2}
\sin\left(\sqrt{2+\sqrt{2}} \, \tau\right) \left(-\sqrt{2}\, \bar
F_1+ \bar F_2\right) d\tau, \vspace{0.2cm}\\
\CG^2(Z_0,W_0)= \displaystyle  \displaystyle \int_0^{pT_2}\cos
\left(\sqrt{2+\sqrt{2}} \, \tau\right) \left(-\sqrt{2}\, \bar F_1+
\bar F_2\right) d\tau,
\end{array}
\end{equation}
where $\bar F_k= F_k(\tau,A_2,B_2,C_2,D_2)$ for $k=1,2$ with
\[
\begin{array}{l}
A_2= -\dfrac{1}{\sqrt{4+2 \sqrt{2}}} \left(Z_0 \cos
\left(\sqrt{2+\sqrt{2}}\, \tau\right)+W_0 \sin
\left(\sqrt{2+\sqrt{2}}\, \tau\right)\right), \vspace{0.2cm}\\
B_2= \dfrac{1}{\sqrt{2}} \left(-W_0 \cos \left(\sqrt{2+\sqrt{2}}\,
\tau\right)+Z_0\sin \left(\sqrt{2+\sqrt{2}}\, \tau\right) \right),
\vspace{0.2cm}\\
C_2= \dfrac{1}{\sqrt{2+\sqrt{2}}} \left(Z_0 \cos
\left(\sqrt{2+\sqrt{2}}\, \tau\right)+W_0
\sin\left(\sqrt{2+\sqrt{2}}\, \tau\right) \right), \vspace{0.2cm}\\
D_2= W_0 \cos \left(\sqrt{2+\sqrt{2}}\, \tau\right)-Z_0 \sin
\left(\sqrt{2+\sqrt{2}}\, \tau\right).
\end{array}
\]
Consider the nonlinear system
\begin{equation}\label{e4a}
\CG^1(Z_0,W_0)=0,\quad \CG^2(Z_0,W_0)=0.
\end{equation}

\smallskip

Our main result on the periodic solutions of the perturbed double
pendulum \eqref{e2} which bifurcate from the periodic solutions of
the unperturbed double pendulum \eqref{e1a} with period $T_2$
traveled $p$ times is the following.

\begin{theorem}\label{t1a}
Assume that the functions $F_k$ of the perturbed double pendulum
with equations of motion \eqref{e2} are periodic in $\tau$ of period
$pT_2/q$ with $p$ and $q$ positive integers relatively prime. Then
for $\e\neq 0$ sufficiently small and for every simple zero
$(Z_0^*,W_0^*)\neq (0,0)$ of the nonlinear system \eqref{e4a}, the
perturbed double pendulum \eqref{e2} has a periodic solution
$(\T_1(\tau,\e),\T_2(\tau,\e))$ tending when $\e\to 0$ to the
periodic solution $(\T_1(\tau),\T_2(\tau))$ run $p$ times of the
unperturbed double pendulum \eqref{e1a} given by
\begin{equation}\label{e5a}
\begin{array}{l}
\left(- \dfrac{1}{\sqrt{4+2 \sqrt{2}}} \left(Z_0^* \cos
\left(\sqrt{2+\sqrt{2}}\, \tau\right)+W_0^* \sin
\left(\sqrt{2+\sqrt{2}}\, \tau\right)\right), \right. \vspace{0.2cm}\\
\left. \dfrac{1}{\sqrt{2+\sqrt{2}}} \left(Z_0^* \cos
\left(\sqrt{2+\sqrt{2}}\, \tau\right)+W_0^*
\sin\left(\sqrt{2+\sqrt{2}}\, \tau\right) \right) \right).
\end{array}
\end{equation}
\end{theorem}

Theorem \ref{t1a} is also proved in section \ref{s3}.

\smallskip

Now we provide an application of Theorem \ref{t1a} in the next
corollary, which will be proved in section \ref{s4}.

\begin{corollary}\label{c2}
If $F_1(\T_1,\T'_1,\T_2,\T'_2)= \T'_2+ \T_1^2 \cos
\left(\sqrt{2+\sqrt{2}}\, t\right)$ and $F_2(\T_1,\T'_1,\T_2,\T'_2)=
0$, then the differential equation \eqref{e2} for $\e\neq 0$
sufficiently small has one periodic solution $(\T_1(\tau,\e),
\T_2(\tau,\e))$ tending when $\e\to 0$ to the periodic solution
$(\T_1(\tau), \T_2(\tau))$ of the unperturbed double pendulum
\eqref{e1a} given by \eqref{e5a} with $(Z_0^*,W_0^*)= \left(0, -8
\left(2+\sqrt{2}\right)\right)$.
\end{corollary}

\section{Proof of Theorems \ref{t1} and \ref{t1a}}\label{s3}

Introducing the variables $(x,y,z,w)= (\T_1,\T'_1,\T_2,\T'_2)$ we
write the differential system of the perturbed double pendulum
\eqref{e2} as a first--order differential system defined in $\R^4$.
Thus we have the differential system
\begin{equation}\label{d1}
\begin{array}{l}
x' = y,\\
y' = -2x+z+\e F_1(\tau,x,y,z,w),\\
z' = w,\\
w' = 2x-2z+\e F_2(\tau,x,y,z,w).
\end{array}
\end{equation}
System \eqref{d1} with $\e=0$ is equivalent to the unperturbed
double pendulum system \eqref{e1a}, called in what follows simply by
the {\it unperturbed system}, otherwise we have the {\it perturbed
system}.

\smallskip

We shall write system \eqref{d1} in such a way that the linear part
at the origin will be in its real normal Jordan form. Then, doing
the change of variables $(\tau,x,y,z,w)\to (\tau,X,Y,Z,W)$ given by
\begin{equation}\label{nn}
\left(
\begin{array}{c}
X\\
Y\\
Z\\
W
\end{array}
\right)= \left(
\begin{array}{cccc}
 \sqrt{1-\dfrac{1}{\sqrt{2}}} & 0 & \dfrac{\sqrt{2-\sqrt{2}}}{2} & 0 \\
 0 & \dfrac{1}{\sqrt{2}} & 0 & \dfrac{1}{2} \\
 -\sqrt{1+\dfrac{1}{\sqrt{2}}} & 0 & \dfrac{\sqrt{2+\sqrt{2}}}{2} & 0 \\
 0 & -\dfrac{1}{\sqrt{2}} & 0 & \dfrac{1}{2}
\end{array}
\right)\left(
\begin{array}{c}
x\\
y\\
z\\
w
\end{array}
\right),
\end{equation}
the differential system \eqref{d1} becomes
\begin{equation}\label{d2}
\begin{array}{l}
X' = \sqrt{2-\sqrt{2}}\, Y, \vspace{0.2cm}\\
Y' = -\sqrt{2-\sqrt{2}}\, X+ \e \dfrac{1}{2} \left(\sqrt{2}\, \tilde
F_1+\tilde F_2\right),\vspace{0.2cm}\\
Z' = \sqrt{2+\sqrt{2}}\, W,\vspace{0.2cm}\\
W' = -\sqrt{2+\sqrt{2}}\, Z+\e \dfrac{1}{2} \left(\tilde
F_2-\sqrt{2}\, \tilde F_1\right),
\end{array}
\end{equation}
where $\tilde F_i(\tau,X,Y,Z,W)=F_i(\tau,\A,\B,\C,\D)$ for $i=1,2$
with
\[
\begin{array}{l}
\A= \dfrac{X}{\sqrt{4-2 \sqrt{2}}}-\dfrac{Z}{\sqrt{2
\left(2+\sqrt{2}\right)}},\vspace{0.2cm}\\
\B= \dfrac{Y-W}{\sqrt{2}},\vspace{0.2cm}\\
\C= \dfrac{X}{\sqrt{2-\sqrt{2}}}+\dfrac{Z}{\sqrt{2+\sqrt{2}}},\vspace{0.2cm}\\
\D= Y+W.
\end{array}
\]
Note that the linear part of the differential system \eqref{d2} at
the origin is in its real normal Jordan form.

\begin{lemma}\label{L1}
The periodic solutions of the differential system \eqref{d2} with
$\e=0$ are
\begin{equation}\label{d3}
\begin{array}{l}
X(\tau)= X_0 \cos \left(\sqrt{2-\sqrt{2}}\, \tau \right) + Y_0 \sin
\left(\sqrt{2-\sqrt{2}}\, \tau \right),\\
Y(\tau)= Y_0 \cos \left(\sqrt{2-\sqrt{2}}\, \tau \right)-X_0 \sin
\left(\sqrt{2-\sqrt{2}}\, \tau \right),\\
Z(\tau)= 0,\\
W(\tau)= 0,
\end{array}
\end{equation}
of period $T_1$, and
\begin{equation}\label{d4}
\begin{array}{l}
X(\tau)= 0,\\
Y(\tau)= 0,\\
Z(\tau)= Z_0 \cos \left(\sqrt{2+\sqrt{2}}\, \tau \right)+W_0 \sin
\left(\sqrt{2+\sqrt{2}}\, \tau \right),\\
W(\tau)= W_0 \cos \left(\sqrt{2+\sqrt{2}}\, \tau\right)-Z_0 \sin
\left(\sqrt{2+\sqrt{2}}\, \tau \right),
\end{array}
\end{equation}
of period $T_2$.
\end{lemma}

\begin{proof}
Since system \eqref{d2} with $\e=0$ is a linear differential system,
the proof follows easily.
\end{proof}

\begin{proof}[Proof of Theorem \ref{t1}]
Assume that the functions $F_k$ of the perturbed double pendulum
with equations of motion \eqref{e2} are periodic in $\tau$ of period
$pT_1/q$ with $p$ and $q$ positive integers relatively prime. Then
we can think that system \eqref{e2} is periodic in $\tau$ of period
$pT_1$. Thinking in this way the differential system and the
periodic solutions \eqref{d3} have the same period $pT_1$.

\smallskip

We shall apply Theorem \ref{tt} of the appendix to the differential
system \eqref{d2}. We note that system \eqref{d2} can be written as
system \eqref{eq:4} taking
\[
{\bf x}=\left(
\begin{array}{c}
X\vspace{0.2cm}\\
Y\vspace{0.2cm}\\
Z\vspace{0.2cm}\\
W
\end{array}
\right), \quad t=\tau, \quad G_0(t,{\bf x})=\left(
\begin{array}{c}
\sqrt{2-\sqrt{2}}\, Y,\vspace{0.2cm}\\
-\sqrt{2-\sqrt{2}}\, X,\vspace{0.2cm}\\
\sqrt{2+\sqrt{2}}\, W,\vspace{0.2cm}\\
-\sqrt{2+\sqrt{2}}\, Z
\end{array}
\right),
\]
\[
G_1(t,{\bf x})=\left(
\begin{array}{c}
0 \vspace{0.2cm}\\
\dfrac{1}{2} \left(\sqrt{2}\, \tilde F_1+\tilde F_2\right) \vspace{0.2cm}\\
0 \vspace{0.2cm}\\
\dfrac{1}{2} \left(\tilde F_2-\sqrt{2}\, \tilde F_1\right)
\end{array}
\right) \quad \mbox{and} \quad G_2(t,{\bf x,\e})=\left(
\begin{array}{c}
0 \vspace{0.2cm}\\
0 \vspace{0.2cm}\\
0 \vspace{0.2cm}\\
0
\end{array}
\right).
\]

We shall study which periodic solutions \eqref{d3} of the
unperturbed system \eqref{d2} with $\e=0$ can be continued to
periodic solutions of the unperturbed system \eqref{d2} for $\e\neq
0$ sufficiently small.

\smallskip

We shall describe the different elements which appear in the
statement of Theorem \ref{tt} in the particular case of the
differential system \eqref{d2}. Thus we have that $\Omega= \R^4$,
$k=2$ and $n=4$. Let $r_1>0$ be arbitrarily small and let $r_2>0$ be
arbitrarily large. We take the open and bounded subset $V$ of the
plane $Z=W=0$ as
\[
V= \{ (X_0,Y_0,0,0)\in \R^4 : r_1< \sqrt{X_0^2+Y_0^2} < r_2 \}.
\]
As usual $\cl(V)$ denotes the closure of $V$. If $\al=(X_0,Y_0)$,
then we can identify $V$ with the set
\[
\{ \al\in \R^2 : r_1< ||\al|| < r_2 \},
\]
here $|| \cdot ||$ denotes the Euclidean norm of $\R^2$. The
function $\be: \cl(V)\to \R^2$ is $\be(\al)= (0,0)$. Therefore, in
our case the set
\[
\mathcal{Z}=\left\{ {\bf z}_{\alpha}=\left( \alpha,
\beta(\alpha)\right),~~\alpha\in \cl(V) \right\}= \{
(X_0,Y_0,0,0)\in \R^4 : r_1\leq \sqrt{X_0^2+Y_0^2} \leq r_2 \}.
\]
Clearly for each ${\bf z}_{\alpha}\in \mathcal{Z}$ we can consider
the periodic solution ${\bf x}(\tau, {\bf z}_{\alpha})= (
X(\tau),Y(\tau)$, $0,0)$ given by \eqref{d3} of period $pT_1$.

\smallskip

Computing the fundamental matrix $M_{{\bf z}_{\alpha}}(\tau)$ of the
linear differential system \eqref{d2} with $\e=0$ associated to the
$T$--periodic solution ${\bf z}_{\alpha}= (X_0,Y_0,0,0)$ such that
$M_{{\bf z}_{\alpha}}(0)$ be the identity of $\R^4$, we get that
$M(\tau)= M_{{\bf z}_{\alpha}}(\tau)$ is equal to
\[
\left(
\begin{array}{cccc}
\cos \left(\sqrt{2-\sqrt{2}}\, \tau\right) & \sin
\left(\sqrt{2-\sqrt{2}}\, \tau\right) & 0 & 0 \\
-\sin \left(\sqrt{2-\sqrt{2}}\, \tau\right) & \cos
\left(\sqrt{2-\sqrt{2}}\, \tau\right) & 0 & 0 \\
0 & 0 & \cos \left(\sqrt{2+\sqrt{2}}\, \tau\right) & \sin
\left(\sqrt{2+\sqrt{2}}\, \tau\right) \\
0 & 0 & -\sin \left(\sqrt{2+\sqrt{2}}\, \tau\right) & \cos
\left(\sqrt{2+\sqrt{2}}\, \tau\right)
\end{array}
\right).
\]
Note that the matrix $M_{{\bf z}_{\alpha}}(\tau)$ does not depend of
the particular periodic solution ${\bf x}(\tau, {\bf z}_{\alpha})$.
Since the matrix
\[
M^{-1}(0)-M^{-1}(pT_1)=\left(
\begin{array}{cccc}
 0 & 0 & 0 & 0 \\
 0 & 0 & 0 & 0 \\
 0 & 0 & 2 \sin ^2\left(\sqrt{2}\, \pi \right) & \sin \left(2
 \sqrt{2}\, \pi \right) \\
 0 & 0 & -\sin \left(2 \sqrt{2}\, \pi \right) & 2 \sin
 ^2\left(\sqrt{2}\, \pi \right)
\end{array}
\right),
\]
satisfies the assumptions of statement (ii) of Theorem \ref{tt}
because the determinant
\[
\left|
\begin{array}{ll}
2 \sin ^2\left(\sqrt{2}\, \pi \right) & \sin \left(2 \sqrt{2}\, \pi
\right) \\
-\sin \left(2 \sqrt{2}\, \pi \right) & 2 \sin ^2\left(\sqrt{2}\, \pi
\right)
\end{array}
\right|= 4 \sin ^2\left(\sqrt{2} \pi \right)\neq 0,
\]
we can apply this theorem to system \eqref{d2}.

\smallskip

Now $\xi: \R^4\to \R^2$ is $\xi(X,Y,Z,W)= (X,Y)$. We calculate the
function
\[
\CG(X_0,Y_0)=\CG(\alpha)=\xi\left( \dfrac{1}{pT_1} \displaystyle
\int _0^{pT_1} M_{{\bf z}_{\alpha}}^{-1}(\tau)G_1(\tau,{\bf
x}(\tau,{\bf z}_{\alpha})) d\tau\right),
\]
and we obtain
\[
\left( \begin{array}{c} \CG_1(X_0,Y_0)\vspace{0.2cm}\\
\CG_2(X_0,Y_0)
\end{array}
\right)= \left(
\begin{array}{c}
\displaystyle -\dfrac{\sqrt{2-\sqrt{2}}}{4 \pi } \int_0^{pT_1}
\sin\left(\sqrt{2-\sqrt{2}} \, \tau\right) \left(\sqrt{2}\, \bar
F_1+ \bar F_2\right) d\tau \vspace{0.2cm}\\
\displaystyle \frac{\sqrt{2-\sqrt{2}}}{4 \pi } \int_0^{pT_1}\cos
\left(\sqrt{2-\sqrt{2}} \, \tau\right) \left(\sqrt{2}\, \bar F_1+
\bar F_2\right) d\tau
\end{array}
\right)
\]
where the functions of $\bar F_k$ for $k=1,2$  are the ones given in
\eqref{jj}. Then, by Theorem \ref{tt} we have that for every simple
zero $(X_0^*,Y_0^*)\in V$ of the system of nonlinear functions
\begin{equation}\label{hh}
\CG_1(X_0,Y_0)=0, \quad \CG_2(X_0,Y_0)=0,
\end{equation}
we have a periodic solution $(X,Y,Z,W)(\tau,\e)$ of system
\eqref{d2} such that
\[
(X,Y,Z,W)(\tau,\e)\to (X_0^*,Y_0^*,0,0)\quad \mbox{as $\e\to 0$.}
\]
Note that system \eqref{hh} is equivalent to system \eqref{e4},
because both equations only differs in a non--zero multiplicative
constant.

\smallskip

Going back through the change of coordinates \eqref{nn} we get a
periodic solution $(x,y,z,w)(\tau,\e)$ of system \eqref{d3} such
that
\[
\left(
\begin{array}{c}
x(\tau,\e) \vspace{0.2cm}\\
y(\tau,\e) \vspace{0.2cm}\\
z(\tau,\e) \vspace{0.2cm}\\
w(\tau,\e)
\end{array}
\right)\to \left(
\begin{array}{c}
\dfrac{1}{\sqrt{4-2 \sqrt{2}}} \left(X_0^* \cos
\left(\sqrt{2-\sqrt{2}}\, \tau\right)+Y_0^* \sin
\left(\sqrt{2-\sqrt{2}}\, \tau\right)\right) \vspace{0.2cm}\\
\dfrac{1}{\sqrt{2}} \left(Y_0^* \cos \left(\sqrt{2-\sqrt{2}}\,
\tau\right)-X_0^*\sin \left(\sqrt{2-\sqrt{2}}\, \tau\right) \right) \vspace{0.2cm}\\
\dfrac{1}{\sqrt{2-\sqrt{2}}} \left(X_0^* \cos
\left(\sqrt{2-\sqrt{2}}\, \tau\right)+Y_0^*
\sin\left(\sqrt{2-\sqrt{2}}\, \tau\right) \right) \vspace{0.2cm}\\
Y_0^* \cos \left(\sqrt{2-\sqrt{2}}\, \tau\right)-X_0^* \sin
\left(\sqrt{2-\sqrt{2}}\, \tau\right)
\end{array}
\right)
\]
as $\e\to 0$.

\smallskip

Consequently we obtain a periodic solution $(\T_1,\T_2)(\tau,\e)$ of
system \eqref{e2} such that
\[
(\T_1,\T_2)(\tau,\e)\to\left(
\begin{array}{c}
\dfrac{1}{\sqrt{4-2 \sqrt{2}}} \left(X_0^* \cos
\left(\sqrt{2-\sqrt{2}}\, \tau\right)+Y_0^* \sin
\left(\sqrt{2-\sqrt{2}}\, \tau\right)\right) \vspace{0.2cm}\\
\dfrac{1}{\sqrt{2-\sqrt{2}}} \left(X_0^* \cos
\left(\sqrt{2-\sqrt{2}}\, \tau\right)+Y_0^*
\sin\left(\sqrt{2-\sqrt{2}}\, \tau\right) \right) \vspace{0.2cm}
\end{array}
\right)
\]
as $\e\to 0$. Hence Theorem \ref{t1} is proved.
\end{proof}

\smallskip

\begin{proof}[Proof of Theorem \ref{t1a}]
This proof is completely analogous to the proof of Theorem \ref{t1}.
\end{proof}

\section{Proof of the two corollaries}\label{s4}

\begin{proof}[Proof of Corollary \ref{c1}]
Under the assumptions of Corollary \ref{c1} the nonlinear system
\eqref{e3} becomes
\[
\begin{array}{l}
\CG_1(X_0,Y_0)= -\dfrac{\pi  \left(X_0^2+3 Y_0^2+8
\left(-2+\sqrt{2}\right)\right)}{8 \left(2-\sqrt{2}
\right)^{3/2}}, \vspace{0.2cm}\\
\CG_2(X_0,Y_0)= -\dfrac{\pi  X_0 Y_0}{4
\left(2-\sqrt{2}\right)^{3/2}}.
\end{array}
\]
This system has the following four solutions
\[
(X_0^*,Y_0^*)= \left(\pm 2 \sqrt{2 \left(2-\sqrt{2}\right)},0
\right) \,\,\, \mbox{and} \,\,\, (X_0^*,Y_0^*)= \left(0, \pm 2
\sqrt{\frac{2}{3} \left(2-\sqrt{2}\right)} \right).
\]
But the solutions which differs in a sign are different initial
conditions of the same periodic solution of the unperturbed double
pendulum. Moreover, it easy to check that these solutions are
simple. So, by Theorem \ref{t1} we only have two periodic solutions
of the perturbed double pendulum. This completes the proof of the
corollary.
\end{proof}

\begin{proof}[Proof of Corollary \ref{c2}]
Under the assumptions of Corollary \ref{c2} the nonlinear system
\eqref{e3a} becomes
\[
\begin{array}{l}
\CG^1(Z_0,W_0)= -\dfrac{\pi  \left(\sqrt{\left(10-7 \sqrt{2}\right)
\left(2+\sqrt{2}\right)} W_0-8\right) Z_0}{4 \sqrt{2
\left(2+\sqrt{2}\right)}}, \vspace{0.2cm}\\
\CG^2(Z_0,W_0)= \dfrac{\pi  \left(\sqrt{2} W_0^2-2 W_0^2-16 W_0+3
\sqrt{2} Z_0^2-6 Z_0^2\right)}{8 \sqrt{2 \left(2+\sqrt{2}\right)}}.
\end{array}
\]
This system only has the following two real solutions
\[
((Z_0^*,W_0^*)= \left(0, -8 \left(2+\sqrt{2}\right)\right) \,\,\,
\mbox{and} \,\,\, (Z_0^*,W_0^*)= (0,0).
\]
the other two solutions are not real. But the solution $(0,0)$ is
not valid. Therefore, by Theorem \ref{t1a} we only have one periodic
solution of the perturbed double pendulum. This completes the proof
of the corollary.
\end{proof}

\section*{Appendix: Basic results on averaging theory}\label{ap}

In this appendix we present the basic result from the averaging
theory that we shall need for proving the main results of this
paper.

\smallskip

We consider the problem of the bifurcation of $T$--periodic
solutions from differential systems of the form
\begin{equation}\label{eq:4}
\dot {\bf x}(t)= G_0(t,{\bf x})+\e G_1(t,{\bf x})+\e^2 G_{2}(t,{\bf
x}, \e),
\end{equation}
with $\e=0$ to $\e\not= 0$ sufficiently small. Here the functions
$G_0,G_1: \R \times \Omega \to \R^n$ and $G_{2}:\R\times \Omega
\times (-\e_0,\e_0)\to \R^n$ are $\CC^2$ functions, $T$--periodic in
the first variable, and $\Omega $ is an open subset of $\R^n$. The
main assumption is that the unperturbed system
\begin{equation}\label{eq:5}
\dot {\bf x}(t)= G_0(t,{\bf x}),
\end{equation}
has a submanifold of periodic solutions. A solution of this problem
is given using the averaging theory.

\smallskip

Let ${\bf x}(t,{\bf z},\e)$ be the solution of the system
\eqref{eq:5} such that ${\bf x}(0,{\bf z},\e)= {\bf z}$. We write
the linearization of the unperturbed system along a periodic
solution ${\bf x}(t,{\bf z},0)$ as
\begin{equation}\label{eq:6}
\dot {\bf y}=D_{\bf x}{G_0}(t,{\bf x}(t,{\bf z},0)){\bf y}.
\end{equation}
In what follows we denote by $M_{\bf z}(t)$ some fundamental matrix
of the linear differential system \eqref{eq:6}, and by
$\xi:\R^k\times \R^{n-k}\to \R^k$ the projection of $\R^n$ onto its
first $k$ coordinates; i.e. $\xi(x_1,\ldots,x_n)= (x_1,\ldots,x_k)$.

\smallskip

We assume that there exists a $k$--dimensional submanifold
$\mathcal{Z}$ of $\Omega$ filled with $T$--periodic solutions of
\eqref{eq:5}. Then an answer to the problem of bifurcation of
$T$--periodic solutions from the periodic solutions contained in
$\mathcal{Z}$ for system \eqref{eq:4} is given in the following
result.

\begin{theorem}\label{tt}
Let $V$ be an open and bounded subset of $\R^k$, and let $\beta:
\cl(V)\to \R^{n-k}$ be a $\CC^2$ function. We assume that
\begin{itemize}
\item[(i)] $\mathcal{Z}=\left\{ {\bf z}_{\alpha}=\left( \alpha,
\beta(\alpha)\right),~~\alpha\in \cl(V) \right\}\subset \Omega$ and
that for each ${\bf z}_{\alpha}\in \mathcal{Z}$ the solution ${\bf
x}(t,{\bf z}_{\alpha})$ of \eqref{eq:5} is $T$--periodic;

\item[(ii)] for each ${\bf z}_{\alpha}\in \mathcal{Z}$ there
is a fundamental matrix $M_{{\bf z}_{\alpha}}(t)$ of \eqref{eq:6}
such that the matrix $M_{{\bf z}_{\alpha}}^{-1}(0)- M_{{\bf
z}_{\alpha}}^{-1}(T)$ has in the upper right corner the $k\times
(n-k)$ zero matrix, and in the lower right corner a $(n-k)\times
(n-k)$ matrix $\Delta_{\alpha}$ with $\det(\Delta_{\alpha})\neq 0$.
\end{itemize}
We consider the function $\CG:\cl(V) \to \R^k$
\begin{equation}\label{eq:7}
\CG(\alpha)=\xi\left( \dfrac{1}{T} \int _0^T M_{{\bf
z}_{\alpha}}^{-1}(t)G_1(t,{\bf x}(t,{\bf z}_{\alpha})) dt\right).
\end{equation}
If there exists $a\in V$ with $\CG(a)=0$ and $\displaystyle{\det
\left( \left({d\CG}/{d\alpha}\right)(a)\right)\neq 0}$, then there
is a $T$--periodic solution $\varphi (t,\e)$ of system \eqref{eq:4}
such that $\varphi(0,\e)\to {\bf z}_a$ as $\e\to 0$.
\end{theorem}

Theorem~\ref{tt} goes back to Malkin \cite{Ma} and Roseau \cite{Ro},
for a shorter proof see \cite{BFL}.

\section*{Acknowledgements}

The first author is partially supported by a MICIIN/FEDER grant
MTM2008--03437, by a AGAUR grant number 2009SGR-0410 and by ICREA
Academia. The second author is partially suported by the grant
FAPESP 2011/03896-0 The third author is partially supported by a
FAPESP--BRAZIL grant 2007/06896--5. The first and third authors are
also supported by the joint project CAPES-MECD grant
PHB-2009-0025-PC

\end{document}